\documentclass[a4paper,12pt]{article}

\usepackage{amsmath,amsthm,amssymb,mathrsfs,latexsym,amsfonts}
\usepackage{graphicx,psfrag,epsfig}
\usepackage{bbm}
\usepackage{a4wide}
\usepackage[english]{babel}
\usepackage[latin1]{inputenc}
\usepackage{authblk}

\numberwithin{equation}{section}

\newtheorem{theorem}{Theorem}[section]
\newtheorem{lemma}[theorem]{Lemma}
\newtheorem{proposition}[theorem]{Proposition}

\newtheorem{question}{Question\rm}

\newtheorem{definition}{Definition\rm}
\newtheorem{conjecture}{Conjecture\rm}

\newtheorem{Main}{Theorem}

\newcounter{paraga}[section]

\newcommand{\N}{\mathbb{N}}
\newcommand{\Z}{\mathbb{Z}}
\newcommand{\Q}{\mathbb{Q}}
\newcommand{\R}{\mathbb{R}}
\newcommand{\C}{\mathbb{C}}

\begin{document}

\def\MP{\,{<\hspace{-.5em}\cdot}\,}
\def\SP{\,{>\hspace{-.3em}\cdot}\,}
\def\PM{\,{\cdot\hspace{-.3em}<}\,}
\def\PS{\,{\cdot\hspace{-.3em}>}\,}
\def\EP{\,{=\hspace{-.2em}\cdot}\,}
\def\PP{\,{+\hspace{-.1em}\cdot}\,}
\def\PE{\,{\cdot\hspace{-.2em}=}\,}
\def\N{\mathbb N}
\def\C{\mathbb C}
\def\Q{\mathbb Q}
\def\R{\mathbb R}
\def\T{\mathbb T}
\def\A{\mathbb A}
\def\Z{\mathbb Z}
\def\demi{\frac{1}{2}}

\begin{titlepage}
  \title{\LARGE{\textbf{Optimal linearization of vector fields on the torus in non-analytic Gevrey classes}}}
  \author{Abed Bounemoura \\
    CNRS - PSL Research University\\
    (Universit{\'e} Paris-Dauphine and Observatoire de Paris)}
\end{titlepage}

\maketitle

\begin{abstract}
We study linear and non-linear small divisors problems in analytic and non-analytic regularity. We observe that the Bruno arithmetic condition, which is usually attached to non-linear analytic problems, can also be characterized as the optimal condition to solve the linear problem in some fixed non quasi-analytic class. Based on this observation, it is natural to conjecture that the optimal arithmetic condition for the linear problem is also optimal for non-linear small divisors problems in any reasonable non quasi-analytic classes. Our main result proves this conjecture in a representative non-linear problem, which is the linearization of vector fields on the torus, in the most representative non quasi-analytic class, which is the Gevrey class. The proof follows Moser's argument of approximation by analytic functions, and uses in an essential way works of Popov, R\"{u}ssmann and P\"{o}schel.
\end{abstract}

\newpage
\tableofcontents
\newpage

\section{Introduction}\label{s1}

The motivation of this work is to try to understand the discrepancy between an elementary remark and a deep theorem. Let $\T:=\R/\Z$,  $\alpha \in \T$ irrational and $(q_n)_{n \in \N}$ the sequence of denominators of the ``best" rational approximations of $\alpha$ (such approximations are given by expansion into continued fractions). Let us say that $\alpha \in \mathcal{R}$ (for R\"{u}ssmann, see~\cite{Rus75}) if $\ln(q_{n+1})=o(q_n)$ as $n$ goes to infinity and $\alpha \in \mathcal{B}$ (for Bruno, see~\cite{Bru71}) if the sequence $q_n^{-1}\ln(q_{n+1})$ is summable; obviously $\mathcal{B} \subsetneq \mathcal{R}$. The remark is that $\alpha \in \mathcal{R}$ if and only if for any real-analytic function $f:\T \rightarrow \R$ with zero average, there exist a real-analytic function $g:\T \rightarrow \R$ satisfying the linear equation
\begin{equation}\label{Yoc01}
g\circ R_\alpha-g=f
\end{equation} 
where $R_\alpha : \T \rightarrow \R$ denotes the rotation by $\alpha$. The deep theorem, due to Yoccoz (see \cite{Yoc02}), states that $\alpha \in \mathcal{B}$ if and only for any real-analytic orientation-preserving circle diffeomorphism $F : \T \rightarrow \T$ with rotation number $\alpha$ and sufficiently close to $R_\alpha$, there exists a real-analytic circle diffeomorphism $\phi : \T \rightarrow \T$ close to the identity satisfying the conjugacy equation
\begin{equation}\label{Yoc02}
 \phi^{-1} \circ F \circ \phi=R_\alpha. 
\end{equation}
The linear equation~\eqref{Yoc01} appears as a ``linearization" of the non-linear equation~\eqref{Yoc02}, hence it may be surprising that the arithmetic condition needed to solve~\eqref{Yoc02} is stronger than the one needed to solve~\eqref{Yoc01}; this reflects the failure of any kind of inverse function theorem (either the classical version or the Nash-Moser version) for this particular problem. Such a discrepancy does not appear in the smooth category; in this case~\eqref{Yoc01} and~\eqref{Yoc02} are solvable if and only if $\alpha \in \mathcal{D}$ (for Diophantine) which can be expressed by the asymptotic condition $\ln (q_{n+1})=O(\ln(q_n))$. In higher dimensions, results in the smooth case are exactly the same (and the proofs as well) but not in the analytic case: Yoccoz's proof relies on a geometric construction which deeply uses the theory of holomorphic functions in one variable and continued fractions for which there is no known good analogues in higher dimension. Yet one can still define a Bruno condition $\mathcal{B}$ in any dimension and R\"{u}ssmann (see~\cite{Rus94}) proved that it is a sufficient condition to solve the higher dimensional analogue of~\eqref{Yoc02}; it is unknown if such a condition is necessary\footnote{According to Yoccoz, it is not (personal communication)}.    

The purpose of this paper is to study these linear and non-linear small divisors problems in any dimension for regularities intermediate between smooth and analytic. Those regularities fall into two basic classes: the quasi-analytic classes, which are made of functions completely determined by their Taylor expansion at just one point exactly like analytic functions, and the others, the non quasi-analytic classes, and we will be mainly concerned with those latter classes. At that point we should stress out that even if one is interested only in analytic problems, non-quasi-analytic classes appear naturally in problems like~\eqref{Yoc01} or~\eqref{Yoc02}. As we shall see below in~\S\ref{s24}, the Bruno condition $\alpha \in \mathcal{B}$, which can be characterized by the solvability of~\eqref{Yoc02} in the analytic case, can also be characterized if one looks at~\eqref{Yoc01} not in the analytic case but in some arbitrary yet fixed non quasi-analytic class; apparently this hasn't been noticed before, even though this is elementary. Another way in which non quasi-analytic classes show up is more classical: for various reasons one may be interested in the regularity of $g=g(\alpha)$ in~\eqref{Yoc01} and $\phi=\phi(\alpha)$ in~\eqref{Yoc02} as functions of $\alpha$ (with $\alpha$ varying in a closed set, so regularity has to be understood in the sense of Whitney). In general such a dependence is not analytic even if the data is analytic, it is always smooth (see~\cite{Ris99}) and most probably, it always belongs to some non-quasi analytic class which depends on the arithmetic properties of $\alpha$; in the best case scenario (when $\alpha$ is Diophantine) such a dependence belongs to some Gevrey non-analytic class (\cite{MS03},~\cite{Pop04}), which are the most studied non quasi-analytic classes. 

Our main result deals with those Gevrey classes, so let us informally state what it implies in the particular case of~\eqref{Yoc01} and~\eqref{Yoc02}. Let $f_n \in \C$, $n \in \Z$, be the Fourier coefficients of an integrable function on $\T$; it is well-known that $f$ is analytic if and only if $|f_n|\leq e^{-s|n|}$ for some $s>0$ and for all $n$ large enough; we shall say it belongs to the $a$-Gevrey class, $0<a\leq 1$, if we have instead the asymptotic inequalities $|f_n|\leq e^{-s|n|^a}$. So $1$-Gevrey is analytic, but $a$-Gevrey functions for $0<a<1$ are non quasi-analytic (indeed, for $0<a<1$ it is not hard to construct explicitly a non-zero $a$-Gevrey function with an arbitrarily small support). It is easy to observe that~\eqref{Yoc01} can be solved if and only if $\ln(q_{n+1})=o(q_n^a)$, while a consequence of the results in~\cite{BF19} and~\cite{LDG19} is that~\eqref{Yoc02} can be solved if $q_n^{-a}\ln(q_{n+1})$ is summable. The latter condition is an adapted Bruno type condition, which is optimal in the analytic case $a=1$ in view of the work of Yoccoz, and one may ask if this is the case when $0<a<1$. The answer is negative, as the following informal theorem shows.

\begin{theorem}
In the $a$-Gevrey class for $0<a<1$, we have $\ln(q_{n+1})=o(q_n^a)$ if and only if\eqref{Yoc01} can be solved if and only if\eqref{Yoc02} can be solved.
\end{theorem}

This is a particular case of Theorem~\ref{th3} that will be stated in~\S\ref{s22}, and is a corollary of our main result Theorem~\ref{th1} which is a more quantitative statement, valid in any dimension and in the continuous setting (the discrete setting can be recovered by the usual suspension-section argument, which is rather straightforward in our context). So for Gevrey non-analytic classes, there are no Bruno type condition; actually this should be true for any reasonable non quasi-analytic class (where by reasonable we mean that such a class should be stable with respect to basic non-linear operations, which, unlike~\eqref{Yoc01}, seems needed if one wants to study~\eqref{Yoc02}). For quasi-analytic classes, the problem is more subtle.

\section{The linear problem}\label{s2}

\subsection{The cohomological equation}\label{s21}

The continuous version of equation~\eqref{Yoc01} in any dimension $n \geq 2$ is as follows. Let $\omega \in \R^n$ a non-resonant vector, meaning that for all non-zero $k\in \Z^n$, the Euclidean inner product $k\cdot \omega$ is non-zero. We denote by $\mathcal{F}$ the space of real formal Fourier series, whose elements are of the form $f=\sum_{k \in \Z^n}f_ke_k$ with $f_k \in \C$ satisfying $\overline{f_k}=f_{-k}$, and where $e_k(x)=e^{2\pi ik\cdot x}$. The problem is to find $g=\sum_{k \in \Z^n}g_ke_k \in \mathcal{F}$ and $c \in \R$ such that
\begin{equation}\label{L}\tag{L}
 X_\omega g=f-c
\end{equation} 
where $X_\omega=\omega$ is the constant vector field on $\T^n$, acting formally on $\mathcal{F}$ by derivation in the direction $\omega$. The equation~\eqref{L}, which is usually called the cohomological equation, is easily solved because it is a linear equation: both $c=c(f)$ and $g=g(f)$ depends linearly on $f$. Necessarily
\[ c=f_0=\int_{\T^n}fdx \]
where $dx$ is the Haar measure on $\T^n$, and solutions of~\eqref{L} are such that $g_0 \in \R$ is arbitrary while for non-zero $k\in\Z^n$:
\begin{equation}\label{Fourier}
g_k=(i2\pi k\cdot \omega)^{-1}f_k.
\end{equation} 
The solution $g$ is thus unique if we normalize the value of $g_0$: we shall always choose $g_0=0$ and refers to $g$ defined this way as the solution of~\eqref{L}. Now if $f$ is regular, in order for $g$ to be regular there will be a competition between the decay of $|f_k|$ and the inevitable growth of $|k\cdot \omega|$ as the norm of $k$ goes to infinity. We shall quantify this by introducing a weight as follows.

\begin{definition}
A weight $\varphi : [0,+\infty) \rightarrow [0,+\infty)$ is a continuous function, normalized by $\varphi(0)=0$, which is non-decreasing and satisfies $\ln (1+t)=O(\varphi(t))$.
\end{definition}

One should keep in mind the following three main examples of weights:
\begin{itemize}
\setlength\itemsep{0em}
\item the ``smooth" weight $\varphi_0(t):=\ln (1+t)$;
\item the ``Gevrey" weight $\varphi_a(t):=t^a/a$, for $0<a<1$;
\item the ``analytic" weight $\varphi_1(t):=t$.
\end{itemize}

We have $\varphi_a(t) \rightarrow \varphi_1(t)$ as $a \rightarrow 1$, whereas for $t \geq 1$, $\varphi_a(t)-1/a=(t^a-1)/a \rightarrow \ln(t)$ as $a \rightarrow 0$, which is equivalent to $\varphi_0(t)$. As we shall see below, our minimal growth requirement on $\varphi$, which can be written again as $\varphi_0(t)=O(\varphi(t))$, is no loss of generality for the problems we are interested in. To a given weight function we will associate a scale of regularity classes and a scale of arithmetic classes, where we shall always use the norm
\[ |x|:=\max_{1\leq i \leq n}|x_i|, \quad x=(x_1,\dots,x_n) \in \R^n. \]

\begin{definition}
Given a weight $\varphi$ and a regularity parameter $r \in \R$, we define a regularity class
\[ \mathcal{F}_r^\varphi:=\{f \in \mathcal{F} \; | \; ||f||_r:=\sup_{k \in \Z^n}|f_k|e^{r\varphi(|k|)}<+\infty \}. \]
\end{definition}

Observe that $\mathcal{F}_r^\varphi$ is always a Banach space, and we have compact inclusions $\mathcal{F}_{r_2}^\varphi \subseteq \mathcal{F}_{r_1}^\varphi$ whenever $r_1\leq r_2$. Let us make some technical comment on the choice of the ``Fourier based" $l_{\infty}$-norm $||\cdot||_r$ we used to define the space $\mathcal{F}_r^\varphi$; we could have easily chosen any $l_{p}$-norm for $p\in [1,+\infty]$:
\[ ||f||_r^p:=\left(\sum_{k \in \Z^n}|f_k|^p e^{pr\varphi(|k|)}\right)^{1/p}. \]
Those norms are non-equivalent and lead to different spaces, yet they are ``comparable" so the (projective and inductive) limits we will consider below are the same. For the linear problem that we shall consider here, the precise choice of such a norm do not make any difference, but this is not the case for non-linear problems. The $l_{1}$-norm has the advantage that it defines a Banach algebra provided the weight $\varphi$ is sub-additive (that is $\varphi(t+s)\leq \varphi(t)+\varphi(s)$) and thus leads to very simple product estimates, while the $l_{2}$-norm has the advantage that it defines a Hilbert space, and the norm is equivalent to a ``space based" $L_2$-norm which are characterized by the growth of the derivatives, provided $f$ is a smooth function. We chose the $l_\infty$-norm since it will be more convenient for the (almost) characterization of such spaces (under assumptions on the weight $\varphi$) by almost analytic extension and analytic approximation (see~\S\ref{s42}). 

For the smooth weight $\varphi_0(t)=\ln(1+t)$, we shall write $\mathcal{F}_r^{\varphi_0}=\mathcal{F}_r^{0}$ and this space can be compared to the space $C^r(\T^n)$ of H\"{o}lder $1$-periodic functions on $\R^n$ as follows
\[ \mathcal{F}_{r+d+\epsilon}^{0}, \subseteq C^r(\T^n) \subset \mathcal{F}_r^{0} \]
for any $\epsilon>0$, hence $\mathcal{F}_r^{0}$ contains sufficiently smooth functions when $r$ is sufficiently large, and so the same remains true in general for $\mathcal{F}_r^{\varphi}$ according to our assumption $\varphi_0(t)=O(\varphi(t))$. 

\begin{definition}
Given a weight $\varphi$ and an arithmetic parameter $\tau>0$, we define an arithmetic class
\[ \mathcal{A}_\tau^\varphi:=\{\omega \in \R^n \; | \; \gamma_\tau^{-1}(\omega):=\sup_{k \in \Z^n\setminus\{0\}}|2\pi k\cdot \omega|^{-1}e^{-\tau\varphi(|k|)}<+\infty \} . \]
\end{definition}

Observe that $\mathcal{A}_\tau^\varphi$ might be the empty set; however the set $\mathcal{A}_\tau^{\varphi_0}=\mathcal{A}_\tau^{0}$ is precisely the set of Diophantine vectors with exponent $\tau$, which is non-empty for $\tau \geq n-1$ and of full Lebesgue measure for $\tau>n-1$. Hence our assumption $\varphi_0(t)=O(\varphi(t))$ ensure that $\mathcal{A}_\tau^\varphi$ is non-empty and of full measure for all $\tau$ large enough, and this explains why our minimal growth assumption is no loss of generality. The following lemma is obvious.

\begin{lemma}\label{lem1}
For any $\omega \in \mathcal{A}_\tau^\varphi \neq \emptyset$ and any $r \in \R$, we have
\[ ||g||_{r-\tau}\leq \gamma_\tau^{-1}(\omega)||f||_r \]
hence $g \in \mathcal{F}_{r-\tau}^\varphi$ provided $f \in \mathcal{F}_{r}^\varphi$.
\end{lemma}

As a side remark, the above lemma holds true (up to a multiplicative constant which depends on the dimension) with $\mathcal{F}_{r}^{0}$ replaced by the space $C^r(\T^n)$, $r>0$ and $r\notin \N$ (this follows from a Paley-Littlewood decomposition, and holds true actually for general Besov spaces allowing $r \in \R$).

\subsection{The projective limit}\label{s22}

Next we look at the projective limit of the scale of Banach spaces $(\mathcal{F}_r^\varphi)_{r>0}$ and its associated arithmetic class
\[  \mathcal{F}_{\infty}^{\varphi} :=\bigcap_{r>0}\mathcal{F}_r^\varphi, \quad \mathcal{A}_{\infty}^{\varphi} :=\bigcup_{\tau>0}\mathcal{A}_\tau^\varphi  \]
which are also characterized by
\[ f \in \mathcal{F}_{\infty}^{\varphi} \Longleftrightarrow \varphi(|k|)=o(\ln (|f_k|^{-1}), \quad \omega \in \mathcal{A}_{\infty}^{\varphi} \Longleftrightarrow\ln (|2\pi k\cdot\omega|^{-1})=O(\varphi(|k|)). \]
Endowed with its natural projective limit topology, $\mathcal{F}_{\infty}^{\varphi}$ is a Fr\'{e}chet space. We shall say that a weight $\varphi$ dominates another weight $\psi$ if
\[ \psi(t)=O(\varphi(t)) \Longleftrightarrow \psi \preceq \varphi  \]
and we obviously have
\[ \psi \preceq \varphi \Longrightarrow \mathcal{F}_{\infty}^{\varphi} \subseteq \mathcal{F}_{\infty}^{\psi}, \quad \mathcal{A}_{\infty}^{\psi} \subseteq \mathcal{A}_{\infty}^{\varphi}. \]
The converse holds true under some restrictions (for sub-additive weights for instance). By definition, an arbitrary weight dominates the smooth weight $\varphi_0$ and since
\[ \mathcal{F}_{\infty}^{0}=C^\infty(\T^n), \quad \mathcal{A}_{\infty}^{0}=\mathcal{D} \] 
where $\mathcal{D}$ is the set of all Diophantine vectors, for an arbitrary weight $\varphi$ we have 
\[ \mathcal{F}_{\infty}^{\varphi} \subseteq C^\infty(\T^n), \quad \mathcal{D} \subseteq \mathcal{A}_{\infty}^{\varphi}.  \]
so $\mathcal{A}_{\infty}^{\varphi}$ is always of full measure. If $\psi \preceq \varphi$ and $\varphi \preceq \psi$ the weights are said to be equivalent and they define, in the above limits, the same regularity and arithmetic classes. Finally, for the analytic weight $\varphi_1(t)=t$, $\mathcal{F}_{\infty}^{\varphi_1}=\mathcal{F}_{\infty}^{1}$ identifies with entire periodic functions. The following lemma is well-known.

\begin{lemma}\label{lem2}
The vector $\omega \in \mathcal{A}_{\infty}^\varphi$ if and only the equation~\eqref{L} can be solved in $\mathcal{F}_{\infty}^{\varphi}$. More precisely, if $\omega \in \mathcal{A}_{\infty}^\varphi$, then for any $f \in \mathcal{F}_{\infty}^{\varphi}$ we have $g \in \mathcal{F}_{\infty}^{\varphi}$ and if $\omega \notin \mathcal{A}_{\infty}^\varphi$, there exists $f \in \mathcal{F}_{\infty}^{\varphi}$ such that $g$ does not belong to  
\[ \mathcal{F}_{-\infty}^\varphi:=\bigcup_{r<0}\mathcal{F}_r^\varphi. \]
\end{lemma}

So the statement says that either we can always solve the equation~\eqref{L} in the same regularity class or else the solution may loose all regularity (observe that $\mathcal{F}_{-\infty}^{0}=\mathcal{F}_{-\infty}^{\varphi_0}$ is precisely the space of periodic distributions). When $\omega \in \mathcal{A}_{\infty}^\varphi$, the assertion follows directly from Lemma~\ref{lem1}. When $\omega \notin \mathcal{A}_{\infty}^\varphi$, we have  
\[ \limsup_{k \in \Z^n\setminus\{0\}} \frac{\ln (|2\pi k\cdot\omega|^{-1}|}{\varphi(|k|)}=+\infty \]
hence for any positive sequence $\tau_j \rightarrow +\infty$ one can find a sequence of integer vectors $k_j \in \Z^n \setminus \{0\}$ with $|k_j| \rightarrow +\infty$ such that $|2\pi k_j \cdot \omega|^{-1} \geq e^{\tau_j\varphi(|k_j|)}$. If we choose
\[  f=\sum_{j} f_{k_j} e_{k_j}, \quad f_{k_j}:=(i 2\pi k_j\cdot \omega)|2\pi k_j\cdot \omega|^{-1/2} \]
then $f \in \mathcal{F}$ and the infinitely many non-zero Fourier coefficients of $f$ and $g$ satisfy
\[ |f_{k_j}|=|2\pi k_j\cdot \omega|^{1/2}  \leq  e^{-(\tau_j/2)\varphi(|k_j|)}, \quad g_{k_j}=|2\pi k_j\cdot \omega|^{-1/2}  \geq e^{(\tau_j/2)\varphi(|k_j|)}  \]
and the conclusion follows.

\subsection{The inductive limit}\label{s23}

Now we look at the inductive limit and its associated arithmetic class
\[  \mathcal{F}_+^{\varphi} :=\bigcup_{r>0} \mathcal{F}_r^\varphi, \quad \mathcal{A}_+^{\varphi} :=\bigcap_{\tau>0} \mathcal{A}_\tau^\varphi  \]
which admit the ``dual" characterizations
\[ f \in \mathcal{F}_{+}^{\varphi} \Longleftrightarrow \varphi(|k|)=O(\ln (|f_k|^{-1}), \quad \omega \in \mathcal{A}_{+}^{\varphi} \Longleftrightarrow\ln (|2\pi k\cdot\omega|^{-1})=o(\varphi(|k|)) \]
and which are again well-defined up to equivalent weights. The inductive limit topology on $\mathcal{F}_{+}^{\varphi}$ is more complicated, but it is still a complete locally convex topological vector space (yet not metrizable). We obviously have the strict inclusions
\[ \mathcal{F}_{\infty}^{\varphi} \subsetneq  \mathcal{F}_+^\varphi, \quad  \mathcal{A}_+^\varphi \subsetneq \mathcal{A}_{\infty}^{\varphi}.  \]
Now we shall say that a weight $\varphi$ strictly dominates another weight $\psi$ if 
\[ \psi(t)=o(\varphi(t)) \Longleftrightarrow \psi \prec \varphi  \]
and we have
\[ \psi \prec \varphi \Longrightarrow \mathcal{F}_{+}^{\varphi} \subset \mathcal{F}_{\infty}^{\psi}, \quad \mathcal{A}_{\infty}^{\psi} \subset \mathcal{A}_{+}^{\varphi}. \]
In particular, if $\varphi_0(t)=o(\varphi(t))$, then $\mathcal{A}_{+}^{\varphi}$ always contains the set of Diophantine vectors, and therefore it is of full measure (but in full generality it might be empty, as it is the case for $\varphi=\varphi_0$). For the analytic weight $\varphi_1(t)=t$, $\mathcal{F}_+^{\varphi_1}=\mathcal{F}_+^{1}$ is precisely the space $\mathcal{O}(\T^n)$ of real-analytic functions on $\T^n$ and the condition $\mathcal{A}_+^{\varphi_1}=\mathcal{A}_+^1$ is exactly R\"{u}ssmann condition $\mathcal{R}$ as it was introduced in~\cite{Rus75}. For the Gevrey weight $\varphi_a(t)=t^a/a$, $\mathcal{F}_+^{\varphi_a}=\mathcal{F}_+^{a}$ is the so-called space of Gevrey functions.

\begin{lemma}\label{lem3}
The vector $\omega \in \mathcal{A}_+^\varphi \neq \emptyset$ if and only the equation~\eqref{L} can be solved in $\mathcal{F}_+^{\varphi}$. More precisely, if $\omega \in \mathcal{A}_+^\varphi$, then for any $f \in \mathcal{F}_{+}^{\varphi}$ we have $g \in \mathcal{F}_{+}^{\varphi}$ and if $\omega \notin \mathcal{A}_+^\varphi$, there exists $f \in \mathcal{F}_{+}^{\varphi}$ such that $g$ does not belong to  
\[ \mathcal{F}_{-}^\varphi:=\bigcap_{r<0}\mathcal{F}_r^\varphi. \] 
\end{lemma}

Again, either we can always solve the equation~\eqref{L} in the same regularity class or else the solution may loose all positive regularity (observe that $\mathcal{F}_{-}^{1}=\mathcal{F}_{-}^{\varphi_1}$ is precisely the space of periodic hyperfunctions in the sense of Kato). The proof is completely analogous to the proof of Lemma~\ref{lem2}: when $\omega \in \mathcal{A}_+^\varphi$ this follows from Lemma~\ref{lem1} and when $\omega \notin \mathcal{A}_+^\varphi$ we have
\[ \limsup_{k \in \Z^n\setminus\{0\}} \frac{\ln (|2\pi k\cdot\omega|^{-1}|}{\varphi(|k|)}>0 \]
and one can use the same example as before replacing the positive sequence $\tau_j \rightarrow +\infty$ by some $\tau>0$. 

\subsection{The Bruno condition}\label{s24}

Finally, we would like to discuss the so-called Bruno condition, which was first introduced by Bruno in~\cite{Bru71} in a linearization problem in the vicinity of a singular point, and then in various other equivalent forms for other linearization problems on a torus by R\"{u}ssmann (see for instance~\cite{Rus80},~\cite{Rus94} and~\cite{Rus01}). This condition appears in non-linear version of Lemma~\ref{lem3} in the analytic case and for $n=2$, it can be actually characterized this way (recall the problem~\eqref{Yoc02} in~\S\ref{s1}); we shall explain how it can also be characterized if one looks at Lemma~\ref{lem2} in an arbitrary non quasi-analytic class, at the projective limit. We should say that a weight is non-quasi analytic if
\[ \int_1^{+\infty}\frac{\varphi(t)}{t^2}dt<\infty \]
and we denote by $NQ$ the set of all non quasi-analytic weights; the terminology is justified by the famous Denjoy-Carleman theorem which identifies the above condition as the necessary and sufficient conditions for non quasi-analyticity (or even more precisely, for the existence of functions with arbitrarily small support). Now following the definition of~\cite{Rus01}, a vector satisfies the Bruno condition if and only if it belongs to $\mathcal{A}^\varphi_1$ for some quasi-analytic weight $\varphi$; indeed the ``approximating function" in the sense of R\"{u}ssmann is nothing but the exponential of a quasi-analytic weight (we point out that R\"{u}ssmann do not require $\ln(1+t)=O(\varphi(t))$, but without such an assumption $\mathcal{A}^\varphi_1$ is always empty so the definition remains unchanged). Using the fact that a weight $\varphi$ is quasi-analytic if and only if $\tau\varphi$ is quasi-analytic for some $\tau>0$, we arrive at the following representation of Bruno vectors
\begin{equation}\label{repB}
\mathcal{B}=\bigcup_{\varphi \in NQ}\mathcal{A}^\varphi_1=\bigcup_{\varphi \in NQ}\bigcup_{\tau>0}\mathcal{A}^\varphi_\tau=\bigcup_{\varphi \in NQ}\mathcal{A}^\varphi_{\infty}.
\end{equation}   
For $n=2$, $\omega=(1,\alpha)$, the elementary equivalence of $\omega \in \mathcal{B}$ and $\alpha \in \mathcal{B}$, as it was defined in~\S\ref{s1}, is proved in~\cite{Rus01} for instance. Lemma~\ref{lem2} allows the following characterization of $\mathcal{B}$. 

\begin{lemma}\label{lem4}
The vector $\omega \in \mathcal{B}$ if and only if there exists $\varphi \in NQ$ such that the equation~\eqref{L} can be solved in $\mathcal{F}_{\infty}^{\varphi}$.
\end{lemma}

The discrepancy between $\mathcal{B} \subsetneq \mathcal{R}$ seems to be related somehow to the existence of non-trivial quasi-analytic class $\mathcal{F}^\psi_\infty$ which strictly contains the analytic class $\mathcal{O}(\T^n)$. Indeed, first observe that when $\varphi$ is non quasi-analytic, we obviously have $\varphi(t)=o(t)$, therefore we have the following inclusions, which are moreover strict:
\[ \mathcal{F}_\infty^\varphi \supsetneq \mathcal{O}(\T^n), \quad \mathcal{A}^\varphi_{\infty} \subsetneq \mathcal{R}.\] 
Taking respectively the intersection and the union over all non quasi-analytic classes one actually obtains 
\[ \bigcap_{\varphi \in NQ} \mathcal{F}_{\infty}^{\varphi}=\mathcal{O}(\T^n),  \quad \mathcal{B}=\bigcup_{\varphi \in NQ} \mathcal{A}_{\infty}^{\varphi} \subsetneq \mathcal{R}. \]
The first equality claims that the intersection of all non quasi-analytic classes is precisely the space of real-analytic functions: this is a rather surprising result first proved in~\cite{Bang} in a different setting (see~\cite{Bom63} for more general results and see also~\cite{Bjo66}, ~\cite{CZ82} and~\cite{BMT90} for results adapted to our setting). On the other hand, the set $\mathcal{B}$, which is the union of $\mathcal{A}_{\infty}^{\varphi}$ over all non quasi-analytic classes, is not equal to $\mathcal{R}$; indeed take any quasi-analytic weight $\psi(t)=o(t)$ for which $\mathcal{O}(\T^n) \subsetneq \mathcal{F}^\psi_\infty$, for instance one may choose $\psi=\psi_1$ with 
\begin{equation}\label{quasi}
\psi_1(t):=t(\log(1+t))^{-1}
\end{equation}
then any vectors in $\mathcal{A}^\psi_\infty$ (which is non-empty) is in $\mathcal{R}$ but not in $\mathcal{B}$.    

As a last remark, we would like to point out that in his first attempts to obtain a linearization result with a Bruno condition in any dimension, R\"{u}ssmann (see~\cite{Rus80}) required the weight to be not only non quasi-analytic but also that $t^{-1}\varphi(t)$ decreases monotonically to zero; the extra condition is the monotonicity requirement but such a (a priori mild) condition prevents the identification of vectors $\omega=(1,\alpha)\in \R^2$ satisfying this condition with Bruno numbers $\alpha$. In fact, when $t^{-1}\varphi(t)$ decreases monotonically to zero, it is easy to show that $\varphi$ is sub-additive; if we let $NQS$ be the set of all non quasi-analytic sub-additive weights, then we have another Bruno type condition
\begin{equation}\label{QA1}
\mathcal{B}^1:=\bigcup_{\varphi \in NQS}\mathcal{A}^\varphi_{\infty} \subsetneq \mathcal{B}.
\end{equation}   
It is proved in~\cite{Bjo66} that for $\varphi \in NQS$, not only we have $\varphi(t)=o(t)$ but also $\varphi(t)=o(\psi_1(t))$ with $\psi_1$ as in~\eqref{quasi} and therefore (see~\cite{Bom63},~\cite{Bjo66}) the intersection of all non quasi-analytic sub-additive classes is no more the real-analytic class, but the larger quasi-analytic class associated to $\psi_1$
\begin{equation}\label{QA2}
\bigcap_{\varphi \in NQS}=\mathcal{F}_+^{\psi_1} \mathcal{F}_{\infty}^{\varphi} \supsetneq \mathcal{O}(\T^n).
\end{equation}
One should point out that in~\eqref{QA1} and~\eqref{QA2}, one could replace $NQS$ by quasi-analytic weights which are either strictly concave or satisfy $t^{-1}\varphi(t) \searrow 0$: the strict concavity implies the monotonicity condition which implies sub-additivity, and quasi-analytic sub-additive weights are equivalent to strict concave weights (see~\cite{Bjo66} and~\cite{PV84} for instance). For non-linear problems, one may largely speculate that in the quasi-analytic class $\mathcal{F}_+^{\psi_1}$, vectors in  $\mathcal{B}^1$ could play the same role as Bruno vectors in the analytic class.  

Finally, let us briefly explain how we will use in the sequel the observation that Bruno vectors can be represented as~\eqref{repB}. For a given $\varphi \in NQ$, this observation implies that non-linear analytic problems (such as~\eqref{Yoc02} described in~\S\ref{s1}) can be solved if $\omega \in \mathcal{A}^\varphi_\tau$ for an arbitrary $\tau$; if one has a suitable control on how the analytic problem can be solved, and if one can characterize functions in $\mathcal{F}^\varphi_r$ in terms of approximation by real-analytic functions then one may expect that $\omega \in \mathcal{A}^\varphi_\tau$ allows also to solve the problem in the larger space $\mathcal{F}^\varphi_r$; clearly there will be a competition in this approach but it does work in the smooth case $\varphi=\varphi_0$ as was proved by Moser in~\cite{Mos66}. The competition in the Gevrey case $\varphi=\varphi_a$ is more subtle (since it has to become singular for $a=1$) but we will show that it works too, and the proof also shows that it should be true for a large class of non quasi-analytic weights (the Gevrey weights allow for a simpler approach and more explicit computations).

\section{A non-linear problem and main result}\label{s3}

\subsection{Linearization of vector fields on the torus}\label{s31}

So now we finally look at the continuous version of the problem~\eqref{Yoc02} in any dimension $n \geq 2$. We consider a smooth vector field $X=X_\omega+F$ on $\T^n$, where $F$ is sufficiently close to zero with respect to a suitable topology, and we wish to conjugate $X$ to $X_\omega$ by a diffeomorphism close to the identity. This is clearly not possible in general as there are obvious topological restrictions, so to circumvent them, the question we ask is whether we can find a constant vector field $X_\lambda=\lambda \in \R^n$ close to zero and a diffeomorphism $\Phi : \T^n \rightarrow \T^n$ close to the identity such that
\begin{equation}\label{V}\tag{V}
 \Phi^*(X-X_\lambda)=X_\omega
\end{equation} 
that is we wish to conjugate the modified vector field $X-X_\lambda$, with an unknown $\lambda \in \R^n$, to $X_\omega$. So given $X=X_\omega+F$, a solution of~\eqref{V} is a couple $(\Phi,\lambda)$ which satisfy the equation~\eqref{V} and which is close to $(\mathrm{Id},0)$: upon normalization (for instance fixing the average of $\Phi-\mathrm{Id}$ to be zero) it can be shown that such a solution, if it exists, is unique. Such a ``modying term" formulation was introduced by Arnold (see~\cite{Arn61}) in this setting, and then generalizes by Moser (see~\cite{Mos67}). If we require the rotation set of $X$ to contains $\omega$, then it is not hard to show that necessarily $\lambda=0$ and this is precisely  the generalization of~\eqref{Yoc02} in any dimension, but we prefer to keep the modifying term since it gives a more flexible formulation which allows to consider arbitrary vector fields $X$ close to $X_\omega$. In general the problem~\eqref{V} is non-linear, in the sense that both $\lambda=\lambda(F)$ and $\Phi-\mathrm{Id}=(\Phi-\mathrm{Id})(F)$ depend on $F$ in a non-linear fashion. There is, however, one simple case in which~\eqref{V} is actually equivalent to the linear problem~\eqref{L} we studied in~\S\ref{s2}; this is the case where $X$ is proportional to $X_\omega$, which can be written
\begin{equation}\label{rep1}
X=X_\omega+P=(1/f)X_\omega
\end{equation} 
for some nowhere vanishing smooth function $f: \T^n \rightarrow \R$ (one may assume $f$ to be close to $1$ to consider this problem as perturbative, but this is not necessary since it is not a perturbative problem). The flow associated to $X$ is nothing but a time-reparametrization of the flow of $X_\omega$, and it is not hard to prove that the unique (upon normalization) solution $(\Phi,\lambda)$ of~\eqref{V} has to be of the form 
\begin{equation}\label{rep2}
\Phi-\mathrm{Id}=g\omega, \quad \lambda=(1-c)\omega, \quad c=\int_{\T^n}fdx
\end{equation}
where $g: \T^n \rightarrow \R$ is the solution of~\eqref{L}. This is essentially due to Kolmogorov (see \cite{Kol53}) in an implicit form and Herman (see~\cite{Her91}) in this more explicit form. It follows that the necessary and sufficient conditions on $\omega$ which were given in respectively Lemma~\ref{lem2} and Lemma~\ref{lem3} to solve~\eqref{L} in respectively $\mathcal{F}_\infty^{\varphi}$ and $\mathcal{F}_+^\varphi$ are automatically necessary (but a priori not sufficient) conditions to solve~\eqref{V} in those regularity classes. In general, since we look at~\eqref{V} in a perturbative setting, linearizing the conjugacy equation at $P=0$, $\Phi=\mathrm{Id}$, yields a linear equation between vector fields, which is nothing but a vector-valued version of~\eqref{L}. The basic question is whether the optimal arithmetic conditions introduced for the linear problem, which we know are necessary for the non-linear problem, are also sufficient.

For the general smooth case, the answer is yes: the vector $\omega \in \mathcal{D}$ if and only if one can solve the problem~\eqref{V}, which is the exact analogue of Lemma~\ref{lem2} (observe that the analogue of Lemma~\ref{lem3} does not make any sense here since the corresponding arithmetic class is empty). In fact, a much more precise statement holds true: if $\omega \in \mathcal{D}^\tau$ for some $\tau\geq n-1$, then if suffices for $P$ to be of class $C^r$ with $r>\tau+1$ ($r\notin \N$) for $\Phi$ to be of class $C^{r-\tau-\epsilon}$ for any $\epsilon>0$. This is almost a non-linear analogue of Lemma~\ref{lem1}, except from the fact that one needs $r>\tau+1$ instead of $r>\tau$ (for $\Phi$ to be at least $C^1$), but the loss of regularity is essentially the same, namely one looses any $\tau'>\tau$ instead of $\tau$.

For the analytic case, the answer is no in general. As we already explained, the Bruno condition $\omega \in \mathcal{B}$ which is stronger than the R\"{u}ssmann condition $\mathcal{R}$, is known to be sufficient to solve~\eqref{V} for any $n\geq 2$ (\cite{Rus94}) and it is known that it is also necessary for $n=2$ (\cite{Yoc02}), hence there is no analogue of Lemma~\ref{lem3} for the analytic weight $\varphi_1(t)=t$ when $n=2$ (and this is most probably the case for any $n \geq 2$). Here again, a more precise statement is true: Yoccoz's example shows in fact that there is no analogue of Lemma~\ref{lem1} since the Bruno condition is optimal for~\eqref{V} no matter how large is the regularity parameter $r>0$ (which here is essentially a width of analyticity). The question of whether a non-linear analogue of Lemma~\ref{lem2} holds true (the regularity class in this case corresponds to entire functions) is open up to our knowledge.    

The main result of this article deals with Gevrey regularity, which are associated to the weight $\varphi_a(t)=t^a/a$ for $0<a<1$, and we shall prove that in this case the answer is again positive. Actually, we shall prove a non-linear perturbative version of Lemma~\ref{lem1} which is Theorem~\ref{th1} below, with however a significant quantitative difference: the loss of derivatives will be essentially $c(a)\tau$ with $1<c(a)<+\infty$; at the formal limit $a\rightarrow 0$ then $c(a)\rightarrow 1$ in concordance with the fact that in this case functions tend to ``behave" as general smooth functions (beware though that there are many smooth functions that do no belong to any Gevrey class) and at the formal limit $a\rightarrow 1$ then $c(a)\rightarrow +\infty$ in concordance with the fact that such a result cannot hold true in the analytic case (beware that other constants, such as the threshold of applicability, will have a singular limit as $a\rightarrow 1$ so the fact that $c(a)\rightarrow +\infty$ is not necessarily unavoidable). Regardless of this, such a statement is still enough to guarantee we have exact  non-linear analogues of Lemma~\ref{lem2} and Lemma~\ref{lem3}. 

\subsection{Main results: the Gevrey case}\label{s32}

To state our result properly for the Gevrey weight $\varphi_a(t)=t^a/a$ for $0<a<1$, we recall that
\begin{equation}\label{space}
\mathcal{F}_r^a:=\mathcal{F}_r^{\varphi_a}=\{f \in \mathcal{F} \; | \; ||f||_r=\sup_{k \in \Z^n}|u_k|e^{r\varphi_a(|k|)}<+\infty \},
\end{equation}
and
\begin{equation}\label{ari}
\mathcal{A}_\tau^a:=\mathcal{A}_\tau^{\varphi_a}=\{\omega \in \R^n \; | \; \gamma_\tau^{-1}(\omega)=\sup_{k \in \Z^n\setminus\{0\}}|2\pi k\cdot \omega|^{-1}e^{-\tau\varphi_a(|k|)}<+\infty \} . 
\end{equation}
In particular, $\mathcal{F}_r^a$ is strictly contained in $C^{\infty}(\T^n)$ for any $r>0$ and $\mathcal{A}_\tau^a$ strictly contains $\mathcal{D}$ for any $\tau>0$. For any smooth vector field $F=(f_1,\dots,f_n) : \T^n \rightarrow \R^n$ , slightly abusing notation we shall still say that $F \in \mathcal{F}_r^a$ if
\[ ||F||_r:=\max_{1 \leq i \leq p }||f_i||_r<+\infty. \] 
This is the main result.

\begin{Main}\label{th1}
Let $0<a<1$, $r_0>r>0$, $\tau>\tau_0>0$, $1<\kappa<2$, assume that
\begin{equation}\label{seuilth}
r>c\tau, \quad c=c(a,\kappa):=32^{a/(1-a)}(\kappa-1)(\kappa^{1-a}-1)^{-1/(1-a)}
\end{equation}
and let $\rho:=r-c\tau>0$ and $\upsilon:=8^{-a/(1-a)}\rho>0$. Then there exist a small positive constant $\varepsilon^*$ and a large positive constant $C$ that depend on $a,r_0,r,\tau,\tau_0,\kappa$ and $n$ such that for any 
\[ X=X_\omega+F, \quad \omega \in \mathcal{A}_{\tau_0}^a, \quad  ||F||_{r_0}=\varepsilon \leq \gamma_{\tau_0}(\omega)\varepsilon^*\]
there exists a vector $\lambda \in \R^n$ and a diffeomorphism $\Phi \in \mathcal{F}_{\nu}^a$, with $\nu:=\upsilon/2$, which solves~\eqref{V} with the estimates
\[ |\lambda| \leq C\varepsilon, \quad ||\Phi-\mathrm{Id}||_{\nu} \leq C(\varepsilon/\gamma_{\tau_0}(\omega))^{\iota}, \quad \iota:=8^{-a/(1-a)}(2r)^{-1}\upsilon. \] 
\end{Main}
 
We shall discuss later on this rather technical statement, especially on the conditions $r_0>r$, $\tau>\tau_0$ and $r>c\tau$ but first we would like to observe that these conditions disappear at the limits, that is when we look at the spaces 
\[ \mathcal{F}_{\infty}^{a} :=\bigcap_{r>0} \mathcal{F}_r^a, \quad \mathcal{F}_{+}^{a} :=\bigcup_{r>0} \mathcal{F}_r^a  \]
and their associated arithmetic classes
\[ \mathcal{A}_{\infty}^{a} :=\bigcup_{\tau>0} \mathcal{A}_{\tau}^a, \quad \mathcal{A}_{+}^{a} :=\bigcap_{\tau>0} \mathcal{F}_{\tau}^a  \]
and lead to the following non-technical statements, which are non-linear perturbative analogues of Lemma~\ref{lem2} and Lemma~\ref{lem3}.

\begin{Main}\label{th2}
The vector $\omega \in \mathcal{A}_{\infty}^a$ if and only the equation~\eqref{V} can be solved in $\mathcal{F}_{\infty}^{a}$.
\end{Main}

\begin{Main}\label{th3}
The vector $\omega \in \mathcal{A}_{+}^a$ if and only the equation~\eqref{V} can be solved in $\mathcal{F}_{+}^{a}$.
\end{Main}

The direct implications for both statements follow from Theorem~\ref{th1}. It is important to observe that the requirement that $F$ be close to zero need not be formulated in the limit topology, one only need to assume the norm of $F$ to be small in some fixed space $\mathcal{F}_{r_0}^{a}$: in Theorem~\ref{th2}, one is given $\tau_0$ (possibly large) then it suffices to choose $r_0=r_0(\tau_0)$ large enough to satisfy~\eqref{seuilth}, and in Theorem~\ref{th3}, one is given $r_0$ (possibly small) then it suffices to choose $\tau_0=\tau_0(r_0)$ small enough to satisfy~\eqref{seuilth}. That this is so is clear for Theorem~\ref{th3}, but no so obvious for Theorem~\ref{th2} yet this is a classical matter (one has to verify that we can impose additional smoothness on $F$ in Theorem~\ref{th1} to get additional smoothness on $\Phi$ with the same smallness condition; see for example Corollary 1 of~\cite{Sal04} in the smooth case, the argument in our case is similar). The converse implication of both statements follows from the converse implication in respectively Lemma~\ref{lem2} and Lemma~\ref{lem3}, together with the case of reparametrized constant vector fields described in~\eqref{rep1} and~\eqref{rep2}.    
 
\subsection{Some comments on Theorem~\ref{th1}}\label{sA}

Let us now comment on the technical conditions $r_0>r$, $\tau>\tau_0$ and $r>c\tau$ in Theorem~\ref{th1}. 
The fact that the initial arithmetic parameter $\tau_0$ deteriorates to $\tau>\tau_0$ is in a sense unavoidable: this comes from the fact that for the non-linear problem~\eqref{V}, the small divisors are not the quantities $|k\cdot \omega|$ but rather $|k|^{-1}|k\cdot \omega|$, or equivalently, the small parameter is not the size of $F$ but rather the size of its differential (of course one could still modify the space~\eqref{space} or the set~\eqref{ari} to take this into account, but we have decided not to do so). In the smooth case, this has the effect of replacing $\tau_0$ by $\tau=\tau_0+1$ but in the Gevrey case, any $\tau>\tau_0$ is sufficient. 
Then the initial regularity parameter $r_0$ also deteriorates to $r<r_0$ but this, however, is artificial: this will be used in order to compensate the fact that, throughout the proof, we shall use various non-equivalent norms and consequently, the constants $\epsilon$ and $C$ becomes singular at the limit $r\rightarrow r_0$. One can in principle work with the same norm along the proof and allows $r=r_0$, but for our non-linear problem, this leads to technical difficulties (changing norms will allow us to somehow avoid those difficulties, at the expense of introducing a singular limit $r\rightarrow r_0$). For some other simpler non-linear problems (that will be mentioned below) one can indeed work with the same norm and reach $r=r_0$ (and also $\tau=\tau_0$). 
Finally let us discuss our assumption $r>c(a,\kappa)\tau$ with 
\[ c(a,\kappa)=32^{a/(1-a)}(\kappa-1)(\kappa^{1-a}-1)^{-1/(1-a)}. \]
As we already mentioned, when $a \rightarrow 0$ then $c(a,\kappa) \rightarrow 1$ and this is the best one can hope, in view of the linear analogue given by Lemma~\ref{lem1}. But when $a \rightarrow 1$, then $c(a,\kappa) \rightarrow +\infty$ at variance with Lemma~\ref{lem1}: here we do not know if this is artificial or not.

To discuss this last issue, let us recall that the smooth analogue of Theorem~\ref{th1}, which we already mentioned, can be proved either by analytic approximations or using Nash-Moser theory, and even though these techniques are clearly related, they are not the same (though they may be equivalent in some sense): in Nash-Moser theory a crucial use is made of the so-called ``tame estimates", whereas they are not used at all if one uses analytic approximations. Our proof of Theorem~\ref{th1} will use analytic approximations, together with the application of an analytic KAM theorem at each step, and it is the analytic KAM theorem which introduces this large constant $c=c(a,\kappa)$. One could, and perhaps should, replace the analytic KAM theorem by an analytic KAM step since this is clearly what is needed; to apply a KAM step one essentially has to require $r>\tau$ but the convergence argument becomes quite technical and heuristic considerations suggest that one needs to ask much more than $r>\tau$ (applying a KAM theorem instead of a KAM step makes the convergence proof elementary). In principle we could also use Nash-Moser theory, albeit in a ``generalized" sense since we only have ``generalized" tame estimates, and we would like to point out that with this approach one faces the same singular limit when $a \rightarrow 1$. First recall that for the Fourier $l_1$-norm
\[ ||f||_r^1=\sum_{k \in \Z^n}|f_k|e^{r\varphi(|k|)}\]
associated to any sub-additive weight (so in particular the smooth weight $\varphi_0$, the Gevrey weight $\varphi_a$ and the analytic weight $\varphi_1$), one has the general product estimate
\[ ||fg||_r^1\leq ||f||_r^1||g||_r^1. \]
Now for the smooth weight $\varphi_0$, we have the better tame estimates
\begin{equation}\label{tame1}
||fg||_r^1\leq C(r) (||f||_r^1||g||_0^1+||f||_0^1||g||_r^1), \quad \varphi_0(t)=\ln(1+t) 
\end{equation}
which are fundamental in Nash-Moser theory; those product estimates are the simplest non-linear tame estimates. For the Gevrey weight $\varphi_a$, one can only prove weaker tame estimates
\begin{equation}\label{tame2}
||fg||_r^1\leq C (||f||_r^1||g||_{\chi(a)r}^1+||f||_{\chi(a)r}^1||g||_r^1), \quad \varphi_a(t)=t^a/a 
\end{equation}
with $0<\chi(a)<1$ such that $\chi(a) \rightarrow 0$ as $a \rightarrow 0$ and $\chi(a) \rightarrow 1$ as $a \rightarrow 1$; in general those estimates are the best possible and for the analytic weight $a=1$, one cannot do better than the general product estimate. Even though the estimates~\eqref{tame2} are weaker than those in~\eqref{tame1}, a generalized Nash-Moser argument applies: this was done in~\cite{LZ79} for instance, and their result leads to exactly the same singular behavior as $a \rightarrow +\infty$. More precisely, in the language of Nash-Moser theory we have a ``non-constant loss of derivatives" for the non-linear tame estimates (we still have nice linear tame estimates as Lemma~\ref{lem1} shows) and the result of~\cite{LZ79} requires $r>c'(a)\tau$ with $c'(a)=O_2\left(1/(1-\chi(a))\right)$ (there $\kappa=3/2$ is fixed). Of course product estimates are not sufficient to deal with the non-linear problem~\eqref{V}, as one also needs composition estimates which are usually harder to obtain, and this is why we eventually used analytic approximations, which, in our opinion, turns out to be simpler (even though they require to use non-equivalent norms and induce artificial small loss of regularity). 
 
\subsection{Some comments on other (non) quasi-analytic classes}\label{s33}

Let us now discuss the case of other weights, starting with non quasi-analytic ones. Again, to deal with non-linear problems such as~\eqref{V}, one should require further properties that ensure stability with respect to non-linear operations such as products and compositions: a general property that guarantee such stability properties is the sub-additivity of the weight. A statement such as Theorem~\ref{th1} seems hard to obtain for an abstract weight, since it depends strongly on the weight and on the choice of a norm for a given regularity parameter: this is not the case for statements such as Theorem~\ref{th2} or Theorem~\ref{th3} which depend only on the equivalence class of the weight and on the topology of the limit spaces, so we shall discuss only those statements. Let us say that a weight $\varphi$ is of moderate growth if there exists $H>1$ such that
\[ \liminf_{t \rightarrow +\infty} \frac{\varphi(Ht)}{\varphi(t)}>1. \]  
The following conjecture seems reasonable.

\begin{conjecture}
Theorem~\ref{th2} and Theorem~\ref{th3} holds true for any non quasi-analytic sub-additive weight of moderate growth.
\end{conjecture}

Indeed, for any non quasi-analytic sub-additive weights, the result of~\S\ref{s42}, which deals with characterization of functions in $\mathcal{F}_r^\varphi$ by their approximation with real-analytic functions, and the result of~\S\ref{s43}, which deals with the analytic KAM theorem for vectors in $\mathcal{A}_r^\varphi$, apply (for the results of~\S\ref{s42}, one needs to change to some equivalent weight and somehow loose a precise control on the regularity parameter, see~\cite{PV84} or~\cite{FNRS}). To be more accurate, the approximation by real-analytic functions and the application of an analytic KAM step (and not KAM theorem) are both governed by the same function, which is the young conjugate of $\varphi$ (see~\S\ref{s42} for a definition): the case of Gevrey weight is simpler since its Young conjugate has an explicit expression, and, more importantly, it also governs the application of the analytic KAM theorem (up to a large factor, a point wich we already discussed). For more general weights, simple examples show that this last point is no longer true so to solve the above conjecture, one probably has to replace the analytic KAM theorem by an analytic KAM step, but heuristic considerations suggest that one can still obtains a convergent scheme (which has to be more involved than what is done in~\S\ref{s44}). The moderate growth condition, which also frequently appears in the literature, is clearly used in the proof in~\S\ref{s44}, but we do not know if this is necessary: again Gevrey weights are homogeneous and this moderate growth property is then obvious and very explicit. Apart from the Gevrey weight, a specific family of weights satisfying the requirements of the above conjecture is given by
\[ \psi_{b}(t)=t(\ln(1+t))^{-b}, \quad b>1. \]

The quasi-analytic case is more subtle, and we should only discuss the analogue of Theorem~\ref{th3} in the most representative case of the weight
\[ \psi_1(t)=t(\ln(1+t))^{-1} \]
which already appeared in~\S\ref{s24}. We can ask the following question.

\begin{question}
What is the necessary and sufficient condition on $\omega$ to solve~\eqref{V} in $\mathcal{F}_{+}^{\psi_1}$?
\end{question}

The condition $\omega \in \mathcal{A}^{\psi_1}_+$, defined in~\S\ref{s23}, is clearly a necessary condition but we do not know if it is also a sufficient condition. The results of~\S\ref{s42} do apply as well, but not the results of~\S\ref{s43} since vectors in $\mathcal{A}^{\psi_1}_{1}$ are not Bruno: but exactly as before, one could replace the analytic KAM theorem by an analytic KAM step (which only requires $\varphi(t)=o(t)$, hence applies to $\varphi=\psi_1$) but unlike non-quasi analytic classes, heuristic considerations do not suggest that such a scheme could converge. 

For this quasi-analytic problem, it may be the case that Bruno type condition appear, like they do for the analytic problem. Results in~\cite{BFmams} do apply here and give the following sufficient condition to solve~\eqref{V} in $\mathcal{F}_{+}^{\psi_1}$:
\begin{equation*}
\omega \in \mathcal{A}^{\varphi}_{1}, \quad \int_{1}^{+\infty}\frac{\ln(1+t)\varphi(t)}{t^2}<+\infty. 
\end{equation*}
One may still improve this condition, perhaps the weaker condition $\omega \in \mathcal{B}^1$ which was introduced in~\S\ref{s24} and reads $\omega \in \mathcal{A}^{\varphi}_{1}$ for some non quasi-analytic sub-additive weight (equivalently, $\varphi$ is non quasi-analytic and satisfy $t^{-1}\varphi(t) \searrow 0$), could play a role here.

\subsection{Some comments on other non-linear problems}\label{s34} 

Finally, we claim that Theorem~\ref{th1} (and hence Theorem~\ref{th2} and Theorem~\ref{th3}) holds true also for other non-linear problems on the torus: we chose the case of perturbation of constant vector fields  since it is the simplest case in which non-trivial arithmetic conditions are known to be necessary. But as it will be clear from the proof (which will be described in~\S\ref{s41}), everything works also for the persistence of Lagrangian quasi-periodic invariant torus in Hamiltonian systems as pioneered by Kolmogorov (\cite{Kol54}), or, even more generally, for the quasi-periodic solutions constructed in~\cite{Mos67}  upon introducing modifying terms: consider this more general setting only introduces further non-essential technicalities that we have decided to avoid. 
 
In particular everything applies as well to the problem of reducibility of elliptic quasi-periodic cocycles, as studied in~\cite{DS75} and~\cite{Rus80} for instance: with modifying terms this is nothing but a ``linear" particular case of~\cite{Mos66} in which the frequency $\omega$ is fixed under perturbation (but some other ``elliptic" frequency moves in a non-linear fashion). This problem is technically simpler as only product estimates are needed (see for instance~\cite{BCL19}) and the proof of the equivalent of Theorem~\ref{th1} in this case greatly simplifies: using a ``generalized" Nash-Moser argument that we already described in~\S\ref{sA}, we can work with the same norm and the statement of Theorem~\ref{th1} holds true with $r=r_0$ and even $\tau=\tau_0$. But more seems to be true: indeed it was proved in~\cite{HY12} (see~\cite{AFK11} for the discrete setting) that for $n=2$, the analogue of Theorem~\ref{th1} also holds true even for the analytic case $a=1$ (with a proper arithmetic condition on the elliptic component): hence there are no Bruno type condition as far as the ``base" frequency $\omega$ is concerned. It is reasonable to expect that there should be no Bruno type condition on $\omega$ for any $n \geq 2$, yet this is an open problem.

\section{Proof of the main result}\label{s4}

\subsection{Strategy of the proof, following Moser}\label{s41}

The proof follows Moser's argument of approximation by real-analytic functions, as described for instance in~\cite{Mos66} in the smooth case. The proof in the smooth case relies on the following two principles. First, finitely differentiable function can be characterized by the rate of approximation by real-analytic functions; if $r$ is the regularity and $s_j$ is the sequence (converging to zero) of analytic widths associated to the analytic approximations, then the optimal rate of convergence is of order $s_j^r$. It is important to observe that those analytic approximations can actually be chosen to be much more than analytic, namely one can chose entire functions of exponential type (which, in the periodic case we are considering, are nothing but trigonometric polynomials) and they are easily obtained by convolution. Second, for an analytic perturbation, with analytic width $s_j$, of a constant Diophantine vector field with exponent $\tau>0$, the threshold of applicability of the analytic KAM theorem is easily seen to be of order $s_j^{\tau+1}$. Combining these two principles one obtains a result in regularity $r>\tau+1$.

In the $a$-Gevrey case for $0<a<1$, we will follow the same principles but they are somehow more complicated in this situation. First, as it was observe in~\cite{Pop04}, approximation of Gevrey functions by real-analytic functions through entire functions of exponential type do not lead to an optimal rate of approximation: if $r$ is the regularity and $s_j$ is the sequence of analytic widths, then the rate of convergence is of order $\exp\left(-r(1/s_j)^a\right)$ and do not characterize $a$-Gevrey functions: one could argue that the space of entire functions of exponential type is too small and do not allow to discriminate the case $0<a<1$ from the analytic case $a=1$, whereas one obviously would like that the rate of approximation tends to infinity as $a$ approaches $1$. The idea of Popov (see~\cite{Pop04}) is to obtain a sequence of analytic approximations in a much more precise way through almost analytic extensions, which are extension to the complex domain for which the $\bar{\partial}$-operator do not vanish (so the extension is not analytic) but vanishes asymptotically with a precise rate as the imaginary part tends to zero. Once we have almost analytic extensions, one can further approximate them to solve the $\bar{\partial}$-operator and this leads to a sequence of real-analytic approximations with an optimal rate of order $\exp\left(-r(r/s_j)^b\right)$ with $b=a/(1-a)$, which does tend to infinity as $a$ tends to one and which does characterizes $a$-Gevrey functions. Second, the threshold of applicability of the analytic KAM theorem with a Bruno frequency vector (recall that frequency vectors $\omega \in \mathcal{A}_\tau^a$ are Bruno) is more subtle, but very precise statements have been given by R\"{u}ssmann (see~\cite{Rus80}, \cite{Rus94}, \cite{Rus01}). In fact, results in~\cite{Rus94} and \cite{Rus01} do apply to an arbitrary Bruno vector (see also~\cite{Pos11},~\cite{BF12}) and it turns out that they are not very well suited for the special Bruno vectors $\omega \in  \mathcal{A}_\tau^a$ we are considering (they lead to an artificial singular limit when $a$ approaches zero). Consequently, we shall rely on a result of P\"{o}schel (\cite{Pos89}) which follows~\cite{Rus80} and gives a quite precise statement for vectors $\omega \in  \mathcal{A}_\tau^a$: the threshold is of order $\exp\left(-c\tau(\tau/s_j)^b\right)$ with $c=c(a)>1$. Again, combining these two principles one obtains a result in regularity $r>c\tau$.

\subsection{Analytic approximation, following Popov}\label{s42}

In this section we recall the results of Popov (see~\cite{Pop04} and also~\cite{PH16}) on the approximation of Gevrey functions by real-analytic functions through almost analytic extension. Some little modifications are required as far as the almost analytic extensions are concerned, since we use Fourier based $l_{\infty}$-norm whereas the results in~\cite{Pop04} use space based $L_{\infty}$-norm (that characterize the regularity of a function through the growth of the sequence of its derivatives). One could simply convert our Fourier norm into a proper space norm, as this would only involves an arbitrarily small loss of regularity parameter (up to some polynomially large factor that can be absorbed by the exponentially small error term that will come into play; anyway such a procedure will be used several times in the sequel). However those results are known in greater generality and they are valid for a large class of weights (even some quasi-analytic weights). The first general result on the existence of almost analytic extension is due to Dynkin (see~\cite{Dyn93} for a survey) still using space based norms; for Fourier norms we shall rely on the results contained in~\cite{PV84} and~\cite{AB03}: results in~\cite{PV84} are stated for functions with one variable but since our weights are ``radial" (in the sense that $\varphi(|k|)$ for $k \in \Z^n$ depends only on $|k| \in \N$) those results applies as well, alternatively results of~\cite{AB03} are stated in any dimension and for general (not necessarily radial) weights. Finally, there are also recent results in~\cite{FNRS} which applies to both space based or Fourier based norms.

Given $v>0$, we consider the complex domain
\[ \T^n_v:=\{\theta=x+iy \in \C^n/\Z^n \; | \ |y| < v \} \]
and for any bounded function $g : \T^n_u \rightarrow \C$, we consider the sup-norm over the complex domain, as well as the sup-norm over the real domain
\[ |g|_v:=\sup_{\theta \in \T^n_v}|g(\theta)|, \quad |g|_0:=\sup_{\theta \in \T^n}|g(\theta)|. \]
Recall that a $C^1$ function $\tilde{f} : \T^n_v \rightarrow \C$ is analytic if and only if it is holomorphic, that is it satisfies a system of Cauchy-Riemann equations
\[ \bar{\partial}_l\tilde{f}(\theta):=1/2(\partial_{x_l}\tilde{f}(\theta)+i\partial_{y_l}\tilde{f}(\theta))=0, \quad 1 \leq l \leq n \]
for any $\theta=(x_1+iy_1,\dots, x_n+iy_n) \in \T^n_v$. Now a real-analytic function $f \in \mathcal{O}(\T^n)=\bigcap_{s>0}\mathcal{F}_s^{1}$ possesses a (unique) analytic extension $\tilde{f}$ to  $\T^n_v$ for some $v>0$: more precisely, one can choose any $v<s$ provided $f \in \mathcal{F}_s^{1}$. This is no longer true for $f \in \mathcal{F}_r^{a}$ with $0<a<1$, but a (non-unique) ``almost" analytic extension always exists (and one can take $v$ as large as we want, though only the case where $v$ is small will be of interest).

To state the result, we shall introduce the Young conjugate of the weight $\varphi_a(t)=t^a/a$, which is defined by
\begin{equation}\label{young}
\varphi_a^*(\xi):=\sup_{t\geq 0}\{\varphi_a(t)-\xi t\}, \quad \xi>  0.
\end{equation}
It is always finite since $\varphi_a(t)=o(t)$, convex, decreasing and we also have
\[ \varphi_a(t)=\inf_{\xi >0}\{\varphi_a^*(\xi)+\xi t\}, \quad t>0 \]
since $\varphi_a$ is concave. Moreover, for any $r>0$ we have
\begin{equation}\label{young2}
r\varphi_a^*(\xi/r)=\sup_{t\geq  0}\{r\varphi_a(t)-\xi t\}, \quad \xi>  0.
\end{equation}
In fact, the form of the Gevrey weight allows for a simple computation of $\varphi_a^*$: define the conjugate exponent $b$ by
\begin{equation}\label{exp}
 0<b:=a/(1-a)<+\infty, \quad 1/a-1/b=1
\end{equation}
then it is easy to check that 
\begin{equation}\label{young3}
\varphi_a^*(\xi)=\frac{1}{b\xi^b}=\varphi_b(1/\xi), \quad \xi>0.
\end{equation}
We have the following statement.

\begin{proposition}\label{dynkin}
Let $f \in \mathcal{F}_{r_0}^{a}$ and $0<r<r_0$. Then $f$ admits a $C^1$ extension $\tilde{f} :\T^n_{r_0} \rightarrow \C$ such that for any $0<v\leq r_0$,
\[ |\tilde{f}|_v \leq C_1||f||_{r_0}, \quad |\bar{\partial}_l\tilde{f}|_v \leq C_1||f||_{r_0}e^{-r\varphi_b(r_0/v)}, \quad 1 \leq l \leq n \]
with a positive constant $C_1=C_1(a,r,r_0,n)$. 
\end{proposition}

This is a consequence of Theorem 2.4 of~\cite{PV84} (see also Theorem 2.2 and Theorem 4.1 of~\cite{AB03}). Actually, the quality of approximation is more precisely an exponentially small factor $e^{-r_0\varphi_b(r_0/v)}$ up to some polynomially large factor $(r_0/v)^d$ with $d=d(a,n)$; we simply decreased $r<r_0$ to absorb this last factor. 

Now for the sequence $v_j=2^{-j}v_0$, $j \in \N$ with $v_0 \leq r_0$, the proposition gives a sequence $\tilde{f}_j :\T^n_{v_j} \rightarrow \C$ which are almost analytic in the sense that $|\bar{\partial}_l\tilde{f}|_{v_j}$ for any $1 \leq l \leq n$ decreases to zero with a stretched exponential speed (in $v_j$) with exponent $b$; in view of~\eqref{exp} we have $b\rightarrow +\infty$ as $a \rightarrow 1$ which agrees with the fact that for $a=1$, the extension can be chosen so that $|\bar{\partial}_l\tilde{f}|_{v_j}$ is identically zero.  

Proposition~\ref{dynkin} constitutes one half of Popov's approximation lemma by real-analytic functions; once we have almost analytic extensions the Proposition 3.1 of~\cite{Pop04} (see also Proposition 2.1 of~\cite{PH16}) yields a sequence of real-analytic approximations $f_j$ with the same stretched exponential speed. However, the approximations $f_j$ have to be constructed on a domain $\T^n_{u_j}$ which is slightly smaller than the domain $\T^n_{v_j}$ on which one has estimates for $\tilde{f}_j$; it suffices to take $u_j<v_j$ (in~\cite{Pop04} the author chooses $u_j=v_j/2$ for simplicity) and so we use again our absorbing factor $r<r_0$ to choose $u_j=(r/r_0)v_j$. Proposition~\ref{dynkin}, together with Popov's approximation lemma, leads to the following statement.

\begin{proposition}\label{popov}
Let $f \in \mathcal{F}_{r_0}^{a}$, $0<u_0 \leq r<r_0$ and $u_j=2^{-j}u_0$ for $j \in \N$. There exists a sequence of real-analytic functions $f_j:\T^n_{u_j} \rightarrow \C$ such that 
\begin{equation*}
\begin{cases}
|f_0|_{u_0}\leq C_2||f||_{r_0}, \\ 
|f_{j+1}-f_{j}|_{u_{j+1}} \leq C_2||f||_{r_0}e^{-r\varphi_b(r/u_j)}, \\
|f_{j}-f|_{0} \leq C_2||f||_{r_0}e^{-r\varphi_b(r/u_j)}
\end{cases}
\end{equation*}
with a positive constant $C_2=C_2(a,r,r_0,n)$. 
\end{proposition}  

Finally, as in~\cite{PH16}, we shall prove that the above proposition admits a partial converse; we shall state it in a way adapted to the proof of Theorem~\ref{th1}.

\begin{proposition}\label{popov2}
Given $r>0$, consider for $j \in \N$ a geometric sequence $w_{j-1}=2^{-j}w_{-1}<r$ and a sequence of real-analytic functions $f^{j-1}:\T^n_{w_{j-1}} \rightarrow \C$ such that $f^{-1}=0$, and assume that
\[ |f^j-f^{j-1}|_{w_{j-1}} \leq e^{-\upsilon\varphi_b(r/w_{j-1})}, \quad w_{j-1}<r, \quad  j \in \N  \]
for some constant $0<\upsilon<r$. Then for any $0<\nu<\upsilon$, if 
\begin{equation}\label{seuilconv}
e^{-(\upsilon-\nu)(2^b-1)\varphi_b(r/w_{-1})} \leq 1/2
\end{equation}
the sequence $f^{j-1}$ converges in $\mathcal{F}^a_\nu$ and the limit $f$ satisfy
\[ ||f||_\nu \leq 2e^{-(\upsilon-\nu)\varphi_b(r/w_{-1})}.  \] 
\end{proposition}  

\begin{proof}
Let us define $h^j=f^j-f^{j-1}$ and expand it into Fourier series $h^j=\sum_{k \in \Z^n}h^j_ke_k$: since $h^j$ is analytic we have
\[ |h_j^k| \leq |h_j|_{w_{j-1}}e^{-w_{j-1}|k|} \leq e^{-\upsilon\varphi_b(r/w_{j-1})}e^{-w_{j-1}|k|}   \] 
so that
\[ \sum_{j \geq 0}||h^j||_\nu=\sum_{j \geq 0}\sup_{k \in \Z^n}|h_k^j|e^{\nu\varphi_a(|k|)} \leq \sum_{j \geq 0} e^{-\upsilon\varphi_b(r/w_{j-1})} \sup_{t \geq 0}\{ e^{\nu\varphi_a(t)-w_{j-1}t}\} \]
therefore from~\eqref{young2} and~\eqref{young3}, we have
\[ \sum_{j \geq 0}||h^j||_\nu \leq \sum_{j \geq 0} e^{-\upsilon\varphi_b(r/w_{j-1})}e^{\nu\varphi_b(\nu/w_{j-1})} \leq  \sum_{j \geq 0} e^{-(\upsilon-\nu)\varphi_b(r/w_{j-1})} < 2e^{-(\upsilon-\nu)\varphi_b(r/w_{-1})} \]
where we used the fact that $\varphi_b(\nu/w_{j-1})\leq \varphi_b(r/w_{j-1})$ since $0<\nu<\upsilon<r$ and $r>w_{j-1}$, and~\eqref{seuilconv} which allows to bound the last series by a geometric series. It follows that the sum of the $h^j$ converges normally in $\mathcal{F}^a_\nu$, and since the latter is a Banach space, the sum converges to some $f$ which is necessarily the limit of $f^{j-1}$ and from the last inequality and the fact that $f^{-1}=0$, at the limit one has the wanted estimate.
\end{proof}

\subsection{Analytic KAM theorem, following R\"{u}ssmann}\label{s43}

In this section, we shall state an analytic KAM theorem adapted to a frequency $\omega \in \mathcal{A}_{\tau_0}^a$, which by definition (recall~\eqref{ari}) satisfies the inequalities
\[ |2\pi k\cdot \omega| \geq \gamma_{\tau_0}(\omega)e^{-\tau_0\varphi_a(|k|)}, \quad k \in \Z^n \setminus\{0\} \]
that can be written as
\begin{equation}\label{arit}
|k\cdot \omega| \geq \frac{\alpha}{\Delta(|k|)}, \quad \Delta(t):=e^{\tau_0\varphi_a(t)}, \quad \alpha:=\gamma_{\tau_0}(\omega)/2\pi, \quad k \in \Z^n \setminus\{0\}. 
\end{equation}
Vectors which satisfy~\eqref{arit} are clearly Bruno vectors, hence the results of R\"{u}ssmann (see~\cite{Rus94},~\cite{Rus01}, see also~\cite{Pos11} and~\cite{BF12}) apply. However those results apply to any Bruno vector $\omega$ and as we will explain below, for the special vectors satisfying~\eqref{arit} they do not give the best quantitative result\footnote{One should point out that, unlike~\cite{Rus01} for instance, the threshold of the main theorem in~\cite{Rus94} is not correct: this comes from a slight mistake in Lemma 3.2 in that reference.}. The difference is that in order to reach a statement valid for all Bruno vectors, it seems that one has to avoid using a ``superlinear" scheme of convergence (Newton method) but rather use a scheme whose ``speed" depends on the arithmetic property of $\omega$. Now for vectors satisfying~\eqref{arit} (or more generally for the set of vectors $\mathcal{B}^1$ we mentioned in~\S\ref{s24} and~\S\ref{s44}), a superlinear scheme is possible and does give better quantitative result. Such an analytic KAM theorem with a superlinear scheme is contained in~\cite{Rus80} where the sup-norm is used and a variant of this scheme (in a much more general setting) is contained in~\cite{Pos89} where the Fourier $l_1$-norm is used; it turns out that using 
the Fourier norm, and consequently the results in~\cite{Pos89}, are much more practical (this will also be explained below).

So to state the main result of~\cite{Pos89} (in the simple setting we are considering), recall that $\Delta$ has been defined in~\eqref{arit}, we now define
\begin{equation}\label{gamma}
\Gamma(\sigma):=\sup_{t\geq 0}\{(1+t)\Delta(t)e^{-\sigma t}\}, \quad \sigma>0
\end{equation}
and for a given $1<\kappa<2$, we set
\begin{equation}\label{psi}
\Psi_\kappa(\sigma):=\inf \prod_{j=0}^{+\infty}\Gamma(\sigma_j)^{\kappa_j}, \quad \kappa_j:=(\kappa-1)\kappa^{-(j+1)}
\end{equation}
where the infimum (which can be shown to be a minimum) is taken over all non-increasing sequences $(\sigma_j)_{j  \in \N}$ whose sum is less or equal than $\sigma$. The fact that~\eqref{gamma} and~\eqref{psi} are indeed finite will be verified later; we shall actually need explicit estimates for them. Given $s>0$ and $f : \T^n \rightarrow \R$ we recall that
\[  ||f||_s^1:=\sum_{k \in \Z^n}|f_k|e^{s|k|}\]
and the assumption that $||f||_s^1<+\infty$ implies that $f$ extends to a holomorphic function on $\T^n_s$. We then extend the definition of the above norm for vector fields on $\T^n$ by taking the maximum norm of each components. Here's the analytic KAM theorem we shall rely on.

\begin{proposition}\label{KAM1}
Let $0<a<1$, $1<\kappa<2$, $0<2\sigma<s$ and for $X=X_\omega+F$ with $\omega$ satisfying~\eqref{arit}, assume that
\[ C_4\alpha^{-1}\Psi_\kappa(\sigma) ||F||_s^1 \leq 1   \]
for some positive numerical constant $C_4$. Then there exists $\lambda \in \R^n$ and $\Phi : \T^n_{s-2\sigma} \rightarrow \T^n_s$ such that 
\begin{equation}\label{mod}
\Phi^*(X-\lambda)=X_\omega
\end{equation}
with the estimates
\[ |\lambda| \leq C_4||F||_s^1, \quad  \max\{||\Phi-\mathrm{Id}||_{s-2\sigma}^1,||D\Phi-\mathrm{Id}||_{s-2\sigma}^1\} \leq C_4\alpha^{-1}\Psi_\kappa(\sigma) ||F||_s^1.   \]
\end{proposition}

This is a direct consequence of Theorem A and the Estimates of Theorem A contained in~\cite{Pos89}, in the very special case where the Hamiltonian is linear with respect to the action variables, there are no elliptic variables (one can put $M=0$ in the above reference) and the Cantor set of frequencies is reduced to a single point. We shall now make several modifications to this result in order to have a statement which will be more convenient in this sequel.

First we shall actually use Proposition~\ref{KAM1} in the case where $s$, and thus $\sigma<s/2$, is a small parameter and in this case, we shall obtain explicit estimates for $\Gamma(\sigma)$ in~\eqref{gamma} and $\Psi_\kappa(\sigma)$ in~\eqref{psi}. Observe that
\[ \Gamma(\sigma)=\sup_{t\geq 0}\{(1+t)e^{\tau_0\varphi_a(t)-\sigma t}\}\]
and the supremum is reached at a value $t_\sigma \rightarrow +\infty$ as $\sigma \rightarrow 0$ and that within this limit, the polynomially large factor $(1+t)$ is dominated by the exponentially large factor $e^{\tau_0\varphi_a(t)}$. Hence given $\tau>\tau_0$, there exists $s^*=s^*(a,\tau,\tau_0)>0$ such that for $0<2\sigma < s \leq s^*$, we have
\[ \Gamma(\sigma)\leq \sup_{t\geq 0}\{e^{\tau\varphi_a(t)-\sigma t}\} \]
and recalling the definition of $\varphi_a^*$ in~\eqref{young}, together with the relations~\eqref{young2} and~\eqref{young3}, this gives
\[ \Gamma(\sigma) \leq e^{\tau\varphi_b(\tau/\sigma)}.  \]
If we had used the sup-norm (as in~\cite{Rus80}) instead of the Fourier $l_1$-norm, the supremum in the definition of $\Gamma(\sigma)$ in~\eqref{gamma} would have to be replaced by a sum (or an integral) and the computations in this case are less explicit: this is the reason why we choose the Fourier $l_1$-norm (sup-norms will be converted into Fourier $l_1$-norms below). Next we shall estimate~$\Psi_\kappa(\sigma)$ in~\eqref{psi}, and actually, we have nothing to do since it is explicitly done in Lemma 6 in~\cite{Pos89}: if we optimize among geometric sequence one is lead to consider
\[\sigma_j:=\sigma(\kappa^{1-a}-1)\kappa^{-(1-a)(j+1)}, \quad j \in \N \] 
which gives the estimate
\begin{equation}\label{Psi}
\Psi_\kappa(\sigma) \leq e^{\delta\tau\varphi_b(\tau/\sigma)}, \quad \delta=\delta(a,\kappa):=(\kappa-1)(\kappa^{1-a}-1)^{-1/(1-a)}. 
\end{equation}
Observe that as $a \rightarrow 1$, then $\delta \rightarrow +\infty$, and as $a \rightarrow 0$, then $\delta \rightarrow 1$ which is the best one can expect: if we had used results valid for any Bruno vectors such as~\cite{Rus01}, we would have find another constant $\delta'$ for which $\delta' \rightarrow +\infty$ as $a \rightarrow 1$ but also as $a \rightarrow 0$, which is clearly not natural.   

Next we convert sup-norms into Fourier $l_1$-norms by using the well-known relations 
\[ |f|_s \leq ||f||_s^1, \quad  ||f||_{s-2\sigma}^1\leq \mathrm{coth}^n \sigma |f|_{s}, \quad 0<2\sigma <s. \]
Using again $\tau>\tau_0$, the polynomial large factor $\mathrm{coth}^n \sigma$ can be absorbed by the exponential large factor in~\eqref{Psi} (strictly speaking, we should introduced yet another $\tau'>\tau>\tau_0$ but clearly one can replace $\tau$ by $\tau_0+(\tau-\tau_0)/2$ and then take $\tau'=\tau$) and it follows that Proposition~\ref{KAM1} holds true with~\eqref{Psi} if we replace the Fourier $l_1$-norm by the sup-norm, require $0<4\sigma<s$ instead of $0<2\sigma<s$ and allow the constant $C_4$ to depends now on $\tau,\tau_0, a$, $\kappa$ and $n$.

Finally, as observed by Moser in~\cite{Mos66}, the modifying term $\lambda$ in~\eqref{mod} need not be constant and it can be replaced by a non-constant modifying term of the form $\Theta^*\lambda$, that is 
\[ \Theta^*\lambda(x):=(D_x \Theta)^{-1}\lambda \]
with
\begin{equation}\label{mod2}
\Theta : \T^n_s \rightarrow \C^n/\Z^n, \quad |D\Theta-\mathrm{Id}|_s \leq 1/3.
\end{equation}
Indeed, the last inequality implies $|(D\Theta)^{-1}-\mathrm{Id}|_s \leq 1/2$  and therefore, for any $\lambda \in \R^n$, we have
\[ |\lambda|/2 \leq |\Phi^*\lambda|_s \leq 3|\lambda|/2. \]
We refer to~\cite{Mos66} for the reduction of this seemingly more general statement to the case where $\Theta=\mathrm{Id}$; clearly this only changes $C_4$ by a numerical factor. With all those modifications, we can finally state a more convenient version of Proposition~\ref{KAM1}.

\begin{proposition}\label{KAM2}
Let $0<a<1$, $\tau>\tau_0>0$, $1<\kappa<2$. There exists $s^*=s^*(a,\tau,\tau_0)$ such that for any $0<4\sigma<s \leq s^*$, the following holds true. Given $X=X_\omega+F$ with $\omega$ satisfying~\eqref{arit} and $\Theta$ satisfying~\eqref{mod2}, if we assume that
\begin{equation}\label{seuil}
C_4 \alpha^{-1} e^{\delta\tau\varphi_b(\tau/\sigma)}|F|_s \leq 1, \quad \delta=(\kappa-1)(\kappa^{1-a}-1)^{-1/(1-a)}   
\end{equation}
for some positive constant $C_4=C_4(a,\kappa,\tau,\tau_0,n)$, then there exist $\lambda \in \R^n$ and $\Phi : \T^n_{s-4\sigma} \rightarrow \T^n_s$ such that 
\begin{equation*}
\Phi^*(X-\Theta^*\lambda)=X_\omega
\end{equation*}
with the estimates and
\[ |\lambda| \leq C_4|F|_s, \quad \max\{|\Phi-\mathrm{Id}|_{s-4\sigma},|D\Phi-\mathrm{Id}|_{s-4\sigma}\} \leq C_4\alpha^{-1}|F|_s e^{\delta\tau\varphi_b(\tau/\sigma)}.   \]
\end{proposition}

\subsection{Proof of Theorem~\ref{th1}}\label{s44}

This section is entirely devoted to the proof of Theorem~\ref{th1}. We recall that we are given $0<a<1$, $r_0>r>0$, $\tau>\tau_0>0$, $1<\kappa<2$, such that
\begin{equation}\label{assum}
r>c\tau, \quad c=c(a,\kappa):=32^b\delta=32^b(\kappa-1)(\kappa^{1-a}-1)^{-1/(1-a)}.
\end{equation}
Our assumption is that $\omega \in \mathcal{A}_{\tau_0}^a$ and
\[||F||_{r_0}:=\varepsilon_0\]
will be required to be sufficiently small. We now define $u_0=u_0(\varepsilon_0)>0$ by the equality
\begin{equation}\label{u_0}
e^{-r\varphi_b(r/(2u_0))}:=C_2C_4\varepsilon_0/\alpha \leq 1/2
\end{equation}  
where $\alpha=\gamma_{\tau_0}(\omega)/2\pi$ and $C_2$ and $C_4$ are the constants appearing in Proposition~\ref{popov} and Proposition~\ref{KAM2} respectively. Then we set
\begin{equation}\label{s_0}
\sigma_0=s_0/8=u_0/16.
\end{equation}
and we define geometric sequences converging to zero
\begin{equation}\label{suites}
\sigma_j:=2^{-j}\sigma_0, \quad s_j:=2^{-j}s_0, \quad u_j:=2^{-j}u_0. 
\end{equation}
We already assumed that $\varepsilon_0$ is small enough so that~\eqref{u_0} is less than $1/2$ (we will require much more than that in the sequel), and in view of~\eqref{s_0}, any further smallness condition on $\varepsilon_0$ is equivalent to a smallness condition on $u_0$ or $\sigma_0$ or $s_0$. Thus we may assume $2u_0<r$ so in particular Proposition~\ref{popov} applies, with the sequence $u_j$ defined in~\eqref{suites}, to each components of $F$ and yields a sequence of analytic vector fields $F_j:\T^n_{u_j} \rightarrow \C$ such that 
\begin{equation}\label{lissage}
\begin{cases}
|F_0|_{u_0}\leq C_2\varepsilon_0, \\
|F_{j+1}-F_{j}|_{u_{j+1}} \leq C_2\varepsilon_0e^{-r\varphi_b(r/u_j)}, \\
|F_{j}-F|_{0} \leq C_2\varepsilon_0e^{-r\varphi_b(r/u_j)}.
\end{cases}
\end{equation}
We wish to apply Proposition~\ref{KAM2} to the analytic vector field $X_0=X_\omega+F_0$ on the domain $\T^n_{s_0}$, which makes sense since $F_0$ is defined on $\T^n_{u_0}$ and $s_0<u_0$, and with $\Theta=\mathrm{Id}$. To do so, first observe that $s_0 \leq s^*$ is yet another smallness condition, $4\sigma_0=s_0/2<s_0$ so that we only need to verify~\eqref{seuil} which is implied by
\begin{equation}\label{averif}
e^{\delta\tau\varphi_b(\tau/\sigma_0)}C_2C_4\varepsilon_0/\alpha=e^{32^b\delta\tau\varphi_b(\tau/2u_0)}e^{-r\varphi_b(r/(2u_0))} \leq e^{-(r-c\tau)\varphi_b(r/(2u_0))}\leq 1
\end{equation}
where we used the equality in~\eqref{u_0}, the fact that $\sigma_0=u_0/16=2u_0/32$, $2u_0<r$ and the definition of $c=32^b\delta$. But our assumption~\eqref{assum} is precisely that $r>c\tau$, therefore the last inequality in~\eqref{averif} holds true because of the inequality in~\eqref{u_0}, consequently Proposition~\ref{KAM2} applies and, since $s_0-4\sigma_0=s_1$, we obtain $\lambda^0 \in \R^n$ and $\Phi^0 : \T^n_{s_1} \rightarrow \T^n_{s_0}$ such that 
\begin{equation*}
(\Phi^0)^*(X_0-\lambda^0)=X_\omega
\end{equation*}
and setting $\rho=r-c\tau>0$, we have the estimates
\[ |\lambda^0| \leq \alpha e^{-r\varphi_b(r/(2u_0))}, \quad  \max\{|\Phi^0-\mathrm{Id}|_{s_1},|D\Phi^0-\mathrm{Id}|_{s_1}\} \leq e^{-\rho\varphi_b(r/(2u_0))}.   \]
Moreover, since $s_0<u_1$, we have $\Phi^0 : \T^n_{s_1} \rightarrow \T^n_{u_1}$.

\textit{Claim}. Set $\lambda^{-1}=0\in \R^n$, $\Phi^{-1}=\mathrm{Id}$ and $u_{-1}=2u_0$. We claim that for any $j \in \N$, there exists $\Phi^j : \T^n_{s_{j+1}} \rightarrow \T^n_{u_{j+1}}$ and $\lambda^{j} \in \R^n$ such that for $X_j=X_\omega+F_j$, we have
\begin{equation}\label{limit}
(\Phi^j)^*(X_j-\lambda^j)=X_\omega
\end{equation}
and
\begin{equation}\label{iter}
\begin{cases}
|\lambda^j-\lambda^{j-1}| \leq \alpha e^{-r\varphi_b(r/u_{j-1})}, \\  \max\{|\Phi^j-\Phi^{j-1}|_{s_{j+1}},|D\Phi^j-D\Phi^{j-1}|_{s_{j+1}}\} \leq e^{-\rho\varphi_b(r/u_{j-1})}. 
\end{cases}
\end{equation}
\begin{proof}[Proof of the claim]
We have already shown the claim to be true for $j=0$, so we may proceed by induction and we assume the statement to holds true for $0 \leq i \leq j$ with $j \in \N$, and we need to show that it remains true for $j+1$. First observe that from~\eqref{iter}, for $u_0$ small enough one has
\begin{equation}\label{dist}
\max\{|\Phi^j-\mathrm{Id}|_{s_{j+1}},|D\Phi^j-\mathrm{Id}|_{s_{j+1}}\} \leq 2e^{-\rho\varphi_b(r/(2u_0))}\leq 1/3. 
\end{equation} 
Then we write
\[ X_{j+1}=X_{\omega}+F_{j+1}=X_j+(F_{j+1}-F_j) \]
and we apply our inductive assumption to get
\begin{equation}\label{vec}
(\Phi^j)^*(X_{j+1}-\lambda^j)=(\Phi^j)^*(X_j-\lambda^j)+(\Phi^j)^*(F_{j+1}-F_j)=X_\omega+G_j 
\end{equation}
where $G_j=(\Phi^j)^*(F_{j+1}-F_j)$ is well-defined since $\Phi^j$ maps $\T^n_{s_{j+1}}$ into $\T^n_{u_{j+1}}$ which is the domain of definition of $F_{j+1}-F_j$; it follows from~\eqref{dist} and~\eqref{lissage} that
\begin{equation}\label{G}
 |G_j|_{s_{j+1}}\leq |D\Phi^j|_{s_{j+1}}|F_{j+1}-F_j|_{u_{j+1}} \leq 2C_2\varepsilon_0 e^{-r\varphi_b(r/u_j)}. 
\end{equation} 
We wish to apply Proposition~\ref{KAM2} to the vector field~\eqref{vec}, with $\Theta=\Phi^j$: it follows from~\eqref{dist} that~\eqref{mod2} holds true, obviously $s_{j+1} \leq s^*$ since this is the case for $j=-1$. Also, in view of~\eqref{G}, to verify~\eqref{seuil} one needs to check that
\begin{equation}\label{averif2}
e^{-\delta\tau\varphi_b(\tau/\sigma_{j+1})} (2C_2C_4\varepsilon_0/\alpha) e^{-r\varphi_b(r/u_j)} \leq e^{-32^b\delta\tau\varphi_b(\tau/u_{j})} e^{-r\varphi_b(r/u_j)} \leq  e^{-\rho\varphi_b(r/u_j)}\leq 1
\end{equation}
where we used the inequality in~\eqref{u_0}, the fact that $\sigma_{j+1}=u_{j+1}/16=u_j/32$ and the definition of $\rho=r-32^b\delta\tau$. Again,~\eqref{averif2} holds true since it holds true for $j=-1$ in view of~\eqref{averif}. Hence Proposition~\ref{KAM2} applies (with $\Theta=\Phi^j$) and gives a vector $\lambda_{j+1}$ and a transformation $\Phi_{j+1}$ with the estimates 
\begin{equation}\label{estim}
|\lambda_{j+1}| \leq \alpha e^{-r\varphi_b(r/u_j)}, \quad \max\{|\Phi_{j+1}-\mathrm{Id}|_{s_{j+2}},|D\Phi_{j+1}-\mathrm{Id}|_{s_{j+2}}\} \leq e^{-\rho\varphi_b(r/u_j)}
\end{equation}
so that
\[ (\Phi_{j+1})^*((\Phi^j)^*(X_{j+1}-\lambda^j)-(\Phi^j)^*\lambda_{j+1})=X_\omega. \]
Hence if we set $\Phi^{j+1}=\Phi^j\circ \Phi_{j+1}$, $\lambda^{j+1}=\lambda^j+\lambda_{j+1}$ then
\begin{equation*}
(\Phi^{j+1})^*(X_{j+1}-\lambda^{j+1})=X_\omega
\end{equation*}
and the estimates~\eqref{iter}, with $j$ replaced by $j+1$, follows from~\eqref{estim} and~\eqref{dist}. Finally, since~\eqref{iter} holds true with $j$ replaced by $j+1$, then~\eqref{dist} as well holds true with $j$ replaced by $j+1$, and since the transformations are real, for $\theta \in \T^n_{s_{j+2}}$ we have
\[ |\mathrm{Im}(\Phi^{j+1}(\theta))|\leq |D\Phi^{j+1}|_{s_{j+2}}|\mathrm{Im}(\theta)| \leq 2s_{j+2}=u_{j+2} \] 
and therefore $\Phi^{j+1}$ maps $\T^n_{s_{j+2}}$ into $\T^n_{u_{j+2}}$, hence the claim is proved.
\end{proof}

To conclude the proof of Theorem~\ref{th1}, we define $w_{j-1}=s_{j+1}=u_{j-1}/8$ for $j \in \N$ so that~\eqref{iter} gives
\[ \max\{|\Phi^j-\Phi^{j-1}|_{w_{j-1}},|D\Phi^j-D\Phi^{j-1}|_{w_{j-1}}\} \leq e^{-\upsilon\varphi_b(r/(w_{j-1}))}, \quad \upsilon:=8^{-b}\rho.  \]
Since $w_{-1}=u_0/4<r$, if we fix $0<\nu=\upsilon/2<\upsilon$, the assumptions of Proposition~\ref{popov2} are satisfied because~\eqref{seuilconv} is yet another smallness condition, so we can apply this proposition to each component of $\Phi^{j-1}-\mathrm{Id}$ and $D\Phi^{j-1}-\mathrm{Id}$ and consequently $\Phi^{j-1}$ converges to some map $\Phi \in \mathcal{F}^a_\nu$, which is necessarily a diffeomorphism in view of the convergence of  $D\Phi^{j-1}$, and satisfy the estimate
\begin{equation}\label{es1}
||\Phi-\mathrm{Id}||_{\nu}\leq 2e^{-(\upsilon-\nu)\varphi_b(r/w_{-1})}=2e^{-8^{-b}(\upsilon-\nu)\varphi_b(r/2u_0)}=2(C_2C_4\varepsilon_0/\alpha)^{\iota}
\end{equation}
with $\iota=8^{-b}(2r)^{-1}\upsilon$, where we used the definition in~\eqref{u_0}, whereas for the $\lambda^{j-1}$, it follows directly from~\eqref{iter} that it converges to some $\lambda \in \R$ which satisfy
\begin{equation}\label{es2}
|\lambda| \leq 2\alpha e^{-r\varphi_b(r/u_{-1})}=2\alpha e^{-r\varphi_b(r/(2u_{0}))}=2C_2C_4\varepsilon_0.
\end{equation}
Since $F_j$ converges uniformly to $F$ by~\eqref{lissage}, $X_j$ converges uniformly to $X=X_\omega+F$ and going to the limit $j \rightarrow +\infty$ in~\eqref{limit} we find
\[ \Phi^*(X-\lambda)=X_\omega\]
and the wanted estimates follows from~\eqref{es1} and~\eqref{es2}.

\bigskip

\textit{Acknowledgements.} The author have benefited from partial funding from the ANR project Beyond KAM.

\addcontentsline{toc}{section}{References}
\bibliographystyle{amsalpha}
\bibliography{NonQA}

\end{document}